\documentclass[11pt]{amsart}
\usepackage[left=3.3cm,top=2.7cm,right=3.3cm]{geometry}
\usepackage{amscd,amssymb,graphicx,subfigure}
\usepackage{enumerate}
\usepackage{pgf,tikz}
\usetikzlibrary{arrows,shapes}
\usetikzlibrary{decorations.pathreplacing}
\usepackage{tkz-berge}
\usepackage{epstopdf}
\usepackage{color}
\usepackage[colorlinks=true, linkcolor=blue, anchorcolor=blue, citecolor=blue, filecolor=blue, menucolor= blue, urlcolor=blue]{hyperref}
\usepackage{verbatim, soul, comment}
\usepackage{mathtools}

\definecolor{patriarch}{rgb}{0.5, 0.0, 0.5}
\newtheorem{thm}{Theorem}

\newtheorem{conj}[thm]{Conjecture}
\newtheorem{lemma}[thm]{Lemma}

\newtheorem{cor}[thm]{Corollary}

\newtheorem*{main}{Main Theorem}

\theoremstyle{definition}
\newtheorem{definition}[thm]{Definition}

\newtheoremstyle{casestyle}
{3pt}
{3pt}
{}
{}
{\bfseries}
{:}
{.5em}
{}

\theoremstyle{casestyle}
\newtheorem{case}{Case}[thm]
\newtheorem{scase}{Subcase}[case]

\newcommand{\cC}{\mathcal{C}}
\newcommand{\cG}{\mathcal{G}}

\newcommand{\cH}{\mathcal{H}}

\newcommand{\R}{\mathcal{R}}

\newcommand{\N}{\mathbb{N}}

\newcommand{\of}{\subseteq}

\renewcommand{\k}{k}
\newcommand{\K}{\mathcal{K}}

\newcommand{\w}{\omega}

\newcommand{\bdy}{\partial}
\newcommand{\bdyell}{\partial^{(\ell)}}
\newcommand{\bdytwo}{\partial^{(2)}}

\DeclarePairedDelimiter\ceil{\lceil}{\rceil}
\DeclarePairedDelimiter\floor{\lfloor}{\rfloor}

\DeclarePairedDelimiter\abs{|}{|}
\DeclarePairedDelimiter\parens{(}{)}
\DeclarePairedDelimiter\set{\{}{\}}

\DeclarePairedDelimiterX\setof[2]{\{}{\}}{#1\,:\,#2}

\title{Many triangles with few edges}%
\author[Kirsch and Radcliffe]{R. Kirsch and A.J. Radcliffe}

\begin{document}
\begin{abstract}Extremal problems concerning the number of independent sets or complete subgraphs in a graph have been well studied in recent years. Cutler and Radcliffe proved that among graphs with $n$ vertices and maximum degree at most $r$, where $n = a(r+1)+b$ and $0 \le b \le r$, $aK_{r+1}\cup K_b$ has the maximum number of complete subgraphs, answering a question of Galvin. Gan, Loh, and Sudakov conjectured that $aK_{r+1}\cup K_b$ also maximizes the number of complete subgraphs $K_t$ for each fixed size $t \ge 3$, and proved this for $a = 1$. Cutler and Radcliffe proved this conjecture for $r \le 6$. 
    
    We investigate a variant of this problem where we fix the number of edges instead of the number of vertices. We prove that $aK_{r+1}\cup \cC(b)$, where $\cC(b)$ is the colex graph on $b$ edges, maximizes the number of triangles among graphs with $m$ edges and any fixed maximum degree $r\le 8$, where $m = a \binom{r+1}{2} + b$ and $0 \le b < \binom{r+1}{2}$.
\end{abstract}

\maketitle
\section{Introduction}

The problem of determining which graphs contain the largest number of complete subgraphs of size $t$ has a long history. One can think of it as starting with the Kruskal-Katona theorem concerning the shadows of uniform hypergraphs.

\begin{thm}[Kruskal \cite{K63}, Katona \cite{K68}]\label{KK}
	Suppose that $1\le \ell\le t$. If $\cH \of \binom{[n]}t$ has size $N$, and we define
	\[
		\bdyell \cH = \setof[\Big]{A \in \binom{[n]}{\ell}}{\text{$\exists\, B\in \cH$ such that $A\of B$}},
	\]
	then $\bdyell \cH$ is at least as large as $\bdyell \cC$, where $\cC \of \binom{[n]}{t}$ consists of the first $N$ $t$-sets in colexicographic (or colex) order, and moreover $\bdyell \cC$ is an initial segment of $\ell$-sets in the colex order.
\end{thm}

This immediately implies, by taking $\ell=2$ and $\cH$ to be the set of complete subgraphs of size $t$ in $G$, that any graph containing $N$ complete subgraphs of size $t$ must have at least $\abs[\big]{\bdy^{(2)}\cC}$ edges, where $\cC$ consists of the first $N$ $t$-sets in colex order. This in turn implies an upper bound on the number of copies of $K_t$ in a graph on $m$ edges.
Similarly, since the optimal graphs are the same for all $t$, the colex graph $\cC(m)$, whose edges are the first $m$ pairs in colex order, has the largest number of complete subgraphs among all graphs with $m$ edges.

These results have been extended in a number of directions, usually by restricting the class of graphs considered. If we write $\k(G)$ for the number of complete subgraphs in a graph $G$, $\k_t(G)$ for the number of complete subgraphs of size $t$, and $\K_t(G)$ for the set of complete subgraphs of size $t$, we wish to find upper bounds on $\k(G)$ and $\k_t(G)$ over the class of graphs satisfying some constraints. One example is the following result due to Zykov \cite{Z49} (see also \cite{E62, H76, R, S71}), which bounds the number of complete subgraphs in graphs with bounded clique number, $\omega(G)$.

\begin{thm}[Zykov \cite{Z49}]\label{thm:Zykov}
	If $t\ge 2$ and $G$ is a graph with $n$ vertices and $\omega(G)\leq \omega$, then 
	\[
		\k_t(G)\leq \k_t(T_{n,\omega}),
	\]
	where $T_{n,\omega}$ is the Tur\'an graph with $\omega$ parts. The extremal graph is unique except when $n<t$ or $\omega<t$. 
\end{thm}

Galvin \cite{G} made the following conjecture\footnote{To be precise, Galvin's conjecture was about maximizing the number of independent sets in a graph of bounded minimum degree, but his conjecture is easily seen to be equivalent to the one here.}, and proved it in a wide range of cases.

\begin{conj}[Galvin \cite{G}]\label{conj:Galvin}
	If $G$ is a graph on $n$ vertices with maximum degree at most $r$, where $r \geq n/2 - 1$, then $\k(G)\leq \k(K_{r+1} \cup K_{n-r-1})$, the union of two complete graphs.
\end{conj}

Cutler and the second author \cite{CR2013} proved this conjecture (indeed, without the lower bound on $r$) showing that the extremal graph is the union of complete graphs $aK_{r+1}\cup K_{b}$ where $n=a(r+1)+b$ and $0\le b < r+1$. Soon thereafter Gan, Loh, and Sudakov \cite{GLS} considered the question of maximizing $\k_t(G)$ over this same class of graphs. They made substantial progress on the following conjecture.

\begin{conj}\label{conj:GLS}
	For all $t\geq 3$ and $n,r\geq 1$, if $G$ is a graph on $n$ vertices with maximum degree at most $r$, then
    \[
        k_t(G) \le k_t(aK_{r+1}\cup K_b),
    \]
    where $n=a(r+1)+b$ and $0\le b<r+1$.
\end{conj}

Note that the conjecture is false for $t=2$, since whenever $n$ is not divisible by $r+1$ the conjectured extremal graph has fewer edges than an $r$-regular graph on $n$ vertices. Gan, Loh, and Sudakov proved Conjecture \ref{conj:GLS} for $a=1$ and also demonstrated that if the conjecture holds for $t=3$, then it holds for all $t \ge 3$.

\subsection{The edge analogue of Conjecture~\ref{conj:GLS}} 
\label{sub:the_edge_analogue_of_conjecture_ref_conj_gls}

In this paper we work on the edge, or Kruskal-Katona, version of Conjecture~\ref{conj:GLS}. We fix the number of edges of $G$ (and allow the number of vertices to be arbitrary) and ask which graphs maximize the number of complete subgraphs of size $t$. 

The edge analogue of Zykov's theorem (Theorem~\ref{thm:Zykov}) is known---it is an immediate consequence of the `rainbow' Kruskal-Katona theorem of Frankl, F\"uredi, and Kalai \cite{FFK88} and a more recent theorem of Frohmader \cite{Frohmader08}. For convenience in stating the theorems it is helpful to make a temporary definition.

\begin{definition}
	A subset $A\of \N$ is \emph{$\w$-rainbow} if no two elements of $A$ are congruent modulo $\w$. We write $\R_\w$ for the collection of all $\w$-rainbow subsets of $\N$.
\end{definition}

We first state the rainbow Kruskal-Katona theorem.
\begin{thm}[Frankl, F{\"u}redi, and Kalai \cite{FFK88}]\label{thm:FFK}
	Suppose that $1\le \ell\le t$. If $\cH \of \binom{\N}t$ has size $N$, and moreover $\cH$ is \emph{$\w$-partite}---i.e., we can partition $\N$ into $\w$ subsets such that no set in $\cH$ contains more than one element from each part---then $\bdyell \cH$ is at least as large as $\bdyell \cC$, where $\cC$ consists of the first $N$ $t$-sets in $\R_\w$ in colex order, and moreover $\bdyell \cC$ is an initial segment of $\R_\w\cap\binom{\N}{\ell}$ in colex order.
\end{thm}

Frohmader's result exploits this theorem to extend its conclusion to flag complexes---set systems defined by the set of complete subgraphs in a graph.
\begin{thm}[Frohmader \cite{Frohmader08}]\label{thm:Froh}
	Let $G$ be a graph having $\w(G)\le \w$. If we let $\cH = \K_t(G)$ then $\bdyell\cH \of \K_\ell(G)$ satisfies the inequality of the previous theorem (though $\cH$ need not be $\w$-partite). 
\end{thm}

\begin{cor}
	If $G$ is a graph with $m$ edges having $\w(G)\le \w$ then for all $t\ge 2$ we have $\k_t(G) \le \k_t(\R_\w(m))$, where $\R_\w(m)$ is the graph whose edges are the first $m$ $2$-sets in $\R_\w$ in colex order. 
\end{cor}

\begin{proof}
	Define $\cC$ to be the first $k_t(G)$ $t$-sets from $\R_\w$ in colex order, and set
	\[
		m' = \abs{\bdytwo \cC}.
	\]
	We must have $m \ge m'$ since, from the fact that $E(G) \supseteq \bdytwo \K_t(G)$, we get
	\[
		m = e(G) \ge \abs{\bdytwo\K_t(G)} \ge \abs{\bdytwo \cC} = m'.
	\]
	Thus 
	\[
		\k_t(R_\w(m)) \ge \k_t(R_\w(m')) = \k_t(G),
	\]
	where the last identity is the simple fact that $\K_t(\bdytwo\cC)=\cC$; every edge of $\bdytwo\cC$ is in a copy of $K_t$. 
\end{proof}


\subsection{Results and Notation} 
\label{sub:notation}

We conjecture the following.

\begin{conj}\label{conj:allt}For any $t \ge 3$, if $G$ is a graph with $m$ edges and maximum degree at most $r$, then 
\[
	k_t(G) \le k_t(aK_{r+1}\cup \cC(b)),
\]
where $m = a\binom{r+1}{2}+b$ and $0 \le b < \binom{r+1}{2}$.
\end{conj}

Conjecture~\ref{conj:allt} is the exact analogue of Conjecture~\ref{conj:GLS}: we build as many $K_{r+1}$'s as we can, and then use our remaining resources optimally. Let's define
\[
	f_t(m,r) = \max\setof{\k_t(G)}{\text{$G$ has $m$ edges and $\Delta(G)\le r$}} .
\]
First note that the conjecture is easily seen to be true asymptotically as $m \to \infty$.

\begin{thm}\label{asymp} For all $3\le t\le r+1$,
	\[
		f_t(m,r) \le m\, \frac{\binom{r+1}{t}}{\binom{r+1}{2}} ,
	\]
	and moreover for fixed $t$ and $r$,
	\[
    	f_t(m,r) = (1 - o(1)) \, m\, \frac{\binom{r+1}{t}}{\binom{r+1}{2}}.
	\]
\end{thm}

\begin{proof}
	For the first bound, note that if $G$ is a graph on $m$ edges with $\Delta(G)\le r$ then the endpoints of an edge $e$ of $G$ have at most $r-1$ common neighbors, and complete subgraphs of size $t$ in $G$ containing $e$ correspond to $K_{t-2}$'s in this set of common neighbors. There are at most $\binom{r-1}{t-2}$ such $K_{t-2}$'s, so, counting pairs $(e,K)$ with $e$ an edge of $G$ and $K \in \K_t(G)$ containing $e$, we have
	\begin{align*}
	    \binom{t}2 \, k_t(G) \le m \, \binom{r-1}{t-2}.
	\end{align*}
	Thus 
	\[
		\k_t(G) \le m\, \frac{\binom{r-1}{t-2}}{\binom{t}2} = m\, \frac{\binom{r+1}{t}}{\binom{r+1}{2}} .
	\]
	Now we have 
	\begin{align*}
		 m\, \frac{\binom{r+1}{t}}{\binom{r+1}{2}} 
		 	&\le (1+o(1)) \parens[\bigg]{\raisebox{0.7ex}{$\displaystyle\frac{m}{\binom{r+1}2}{}-1$}} \,\binom{r+1}t \\
			&\le (1+o(1)) \floor[\bigg]{\raisebox{1ex}{$\displaystyle\frac{m}{\binom{r+1}{2}}$}}\,\binom{r+1}t \\
			&\le (1+o(1)) f_t(m,r) ,
	\end{align*}
where the final inequality comes from considering the graph that is the disjoint union of $\floor[\Big]{\raisebox{0.6ex}{$\frac{m}{\binom{r+1}{2}}$}}$ copies of $K_{r+1}$ and a matching to make the edge count up to $m$.
\end{proof}

We were not able to show that it is sufficient to prove Conjecture~\ref{conj:allt} only for $t=3$. Our main result is that the conjecture is true for triangles ($t=3$) for $r\le 8$. 

\begin{main}If $G$ is a graph with $m$ edges and maximum degree at most $r$ for any fixed $r \le 8$, then $$k_3(G) \le k_3(aK_{r+1}\cup \cC(b)),$$ where $m = a\binom{r+1}{2}+b$ and $0 \le b < \binom{r+1}{2}$. That is, the graphs with the maximum number of triangles consist of as many disjoint copies of $K_{r+1}$ as possible, with the remaining edges formed into a colex graph.\end{main}

Most of our graph theory notation is standard; see for instance Bollob\'as \cite{BBMGT} for a reference. In particular we will write (as we have done above) $G\cup H$ for the disjoint union of $G$ and $H$, and also $nG$ for the disjoint union of $n$ copies of $G$. 

We write $\cG(m,r)$ for the set of graphs $G$ with $m$ edges and having $\Delta(G) \le r$. In this class it will be handy to single out the connected ones; we write $\cG_C(m,r)$ for these. 

In Section~\ref{sec:disco} we prove some general results saying that in proving the conjecture we may restrict our attention to connected graphs that achieve maximum degree $r$. In Section~\ref{sec:clusters} we introduce an approach parallel to the folding technique in \cite{CR2013}, and in Section~\ref{sec:red} we discuss how this restricts the class of potentially extremal graphs. In Section~\ref{sec:seq} we give constraints on the extremal graphs in terms of their degree multisets. Finally, in Section~\ref{sec:mainthm}, we combine these two approaches to prove the main theorem.

\section{Disconnected graphs and graphs with small maximum degree}\label{sec:disco}

In this section, we prove two results that hold for all $r$ and allow us to restrict our attention to connected graphs with maximum degree equal to $r$. Both are corollaries of the following lemma concerning colex graphs. The colex graph $\cC(b)$ consists of a complete graph of size $c$, where $\binom{c}2 \le b < \binom{c+1}2$, and then potentially one more vertex, joined to $d$ vertices of the $K_c$, where $d = b - \binom{c}2$. Given this structure it is often useful to think of $b$ as written in the form $b = \binom{c}2 + d$ where $0\le d<c$. We abbreviate this fact as $b = [c,d]$, and write $\cC(c,d)$ for $\cC([c,d])$. It is easy to check that 
\[
	\k_t(\cC(c,d)) = \binom{c}t + \binom{d}{t-1}.
\]
We in fact also allow $d=c$: we have $[c,c]=[c+1,0]$, and the above formula for $k_t$ still applies.

Thus, the number of $K_t$'s in the conjectured extremal graph, which we denote by $g_t(m,r)$, can be written as
\[
	g_t(m,r) = k_t(aK_{r+1}\cup \cC(b)) = a\binom{r+1}t + \binom{c}t + \binom{d}{t-1},
\]
where $m = a \binom{r+1}2 + b$ and $b=[c,d]$. 

\begin{lemma}\label{lem:b1b2}
	Suppose $1 \le b_i \le \binom{r+1}{2}-1$ for $i=1,2$ and $t \ge 3$. Letting $G$ be the graph $\cC(b_1)\cup\cC(b_2)$ then $k_t(G) < g_t(b_1+b_2,r)$, unless $b_1 = \binom{c_1}{2}$ for some $c_1 \in \mathbb N$ and $b_2 = 1$ (or vice versa), in which case $k_t(G) = k_t(\cC(b_1+b_2)) = g_t(b_1+b_2,r) = f_t(b_1+b_2,r)$.
\end{lemma}

\begin{proof} Let's write $b_i = [c_i,d_i]$ for $i = 1, 2$ with $0\le d_i<c_i$ and $b_1 \ge b_2$ (so $c_1 \ge c_2$).  We split into cases depending on the values of the $d_i$.
	
\begin{case} $d_1, d_2 \ge 1$.
\end{case}

If $d_1 < d_2$, then since $b_1 \ge b_2$ we have $c_1 > c_2$, and $d_1 < d_2 < c_2 < c_1$. Then
\begin{align*}
\binom{c_1}{2} + d_2 &< \binom{c_1}{2} + c_1 = \binom{c_1+1}{2}\\
\binom{c_2}{2} + d_1 &< \binom{c_2}{2} + c_2 = \binom{c_2+1}{2},
\end{align*}
and we may swap $d_1$ and $d_2$ as $k_t(\cC(b_1)\cup\cC(b_2)) = k_t(\cC(c_1,d_2)\cup\cC(c_2,d_1))$, with $[c_1,d_2] > [c_2,d_1]$. Therefore we may assume $d_1 \ge d_2$.

First note that $\cC(b_1+1)\cup \cC(b_2-1) \in \cG(m,r)$ since $b_1 \le \binom{r+1}{2}-1$. We compare the number of $K_t$'s in $\cC(b_1)\cup\cC(b_2)$ to those in $\cC(b_1+1)\cup \cC(b_2-1)$. Observe that $k_t(\cC(b_1+1)) = k_t(\cC(b_1)) + \binom{d_1}{t-2}$ and $k_t(\cC(b_2-1)) = k_t(\cC(b_2)) - \binom{d_2-1}{t-2}$, so 
\[
	k_t(\cC(b_1+1)\cup\cC(b_2-1)) = k_t(\cC(b_1)\cup\cC(b_2)) + \binom{d_1}{t-2} - \binom{d_2-1}{t-2} > k_t(\cC(b_1)\cup\cC(b_2))
\] since $d_1 > d_2-1$.

\begin{case} Exactly one of $d_1,d_2$ is zero. 
\end{case}

\begin{scase}\label{sc:mid} $d_2 \ne 0$ or $d_1 < c_2$\end{scase}

If $d_2 = 0$ and $d_1 < c_2$,
\[
	\k_t(\cC(c_1,d_1)\cup\cC(c_2,0)) = \k_t(\cC(c_1,0)\cup\cC(c_2,d_1)),
\]
so we may assume that it is $d_1$ that is zero, and $d_2 \ge 1$.

We compare $G$ to the graph 
\[
	G' = \cC(b_1+c_2) \cup \cC(b_2-c_2) = \cC(c_1,c_2)\cup \cC(c_2-1,d_2-1) \in \cG(b_1+b_2,r).
\]
Note first that there are enough edges in $\cC(b_2)$ to remove $c_2$ of them because $c_2 \ge d_2+1 \ge 2$, and $d_2 \ge 1$, so if $c_2 = 2$ then $b_2 = 2$, and if $c_2 \ge 3$ then $b_2 > \binom{c_2}{2} \ge c_2$. Note also that $\Delta(\cC(b_1+c_2)) \le r$ because $d_1=0$ implies $b_1+c_2 \le \binom{r}{2}+r = \binom{r+1}2$. To prove the second equality above we have 
\[
	[c_2,d_2] - c_2 = \binom{c_2}{2} + d_2 - c_2 = \binom{c_2}{2} - \binom{c_2-1}{1} + d_2-1 = \binom{c_2-1}{2} + d_2 -1 = [c_2-1, d_2 -1].
\]
Note that in the representation $\cC(c_1,c_2)\cup \cC(c_2-1,d_2-1)$ we might have $c_2=c_1$.
The net change in $\k_t$ is
\begin{align*}
 k_t(G') - k_t(G)
 		&= \binom{c_1}{t} + \binom{c_2}{t-1} + \binom{c_2-1}{t} + \binom{d_2-1}{t-1} - \left(\binom{c_1}{t} + \binom{c_2}{t} + \binom{d_2}{t-1}\right)\\
		&= \binom{c_2}{t-1} + \binom{c_2-1}{t} - \binom{c_2}{t} + \binom{d_2-1}{t-1} - \binom{d_2}{t-1}\\
 	    &= \binom{c_2}{t-1} - \binom{c_2-1}{t-1} - \left(\binom{d_2}{t-1} - \binom{d_2 -1}{t-1}\right)\\
 	    &= \binom{c_2-1}{t-2} - \binom{d_2-1}{t-2}\\
 	    & > 0
\end{align*} because $c_2 > d_2$.

\begin{scase}$d_2 = 0$ and $d_1 \ge c_2$\end{scase}

We compare $G$ 
to the graph
\[
	G' = \cC(b_1+1)\cup\cC(b_2-1) = \cC(c_1,d_1+1)\cup\cC(c_2-1,c_2-2) \in \cG(b_1+b_2,r).
\]

To prove the second equality, we have
\[
	b_2-1 = \binom{c_2}{2}-1 = \binom{c_2-1}{2}+\binom{c_2-1}{1} - 1 = \binom{c_2-1}{2}+c_2-2.
\]

Note that we might have $[c_1,d_1+1] = [c_1+1,0]$. The net change in $\k_t$ is
\begin{align*}
 k_t(G') - k_t(G)
		&=\binom{c_1}t + \binom{d_1+1}{t-1} - \binom{c_1}t - \binom{d_1}{t-1} + \binom{c_2-1}t +\binom{c_2-2}{t-1}-\binom{c_2}t\\
		&=\binom{d_1+1}{t-1} - \binom{d_1}{t-1} -\left(\binom{c_2}t- \binom{c_2-1}t\right) +\binom{c_2-2}{t-1}\\
		&=\binom{d_1}{t-2} -\left(\binom{c_2-1}{t-1}- \binom{c_2-2}{t-1}\right)\\
		&=\binom{d_1}{t-2} - \binom{c_2 - 2}{t-2}\\
		& > 0
\end{align*} because $d_1 > c_2-2$.


\begin{case}$d_1 = d_2 =0$\end{case}

If $c_2 = 2$, then $b_2 = 1$. In this case, $\cC(b_2)$ is a single edge, and $\cC(b_1)$ is a complete graph, so 
\[
	k_t(\cC(b_1)\cup\cC(b_2)) = k_t(\cC(b_1+b_2)) = f_t(b_1+b_2,r)
\]
by the Kruskal-Katona theorem.

Otherwise, $c_2 \ge 3$, so $b_2 = \binom{c_2}{2} \ge c_2$, and we compare $G$ to 
\[
	G' = \cC(b_1+c_2) \cup \cC(b_2-c_2) = \cC(c_1,c_2)\cup \cC(c_2-2,c_2-3) \in \cG(b_1+b_2,r).
\] 

As in Subcase \ref{sc:mid}, $\Delta(\cC(b_1+c_2))\le r$. To prove the second equality, we have
\[
	b_2-c_2 = \binom{c_2}2-c_2 = \left( \binom{c_2-2}2+\binom{c_2-2}1\right)+\binom{c_2-1}1-c_2 = \binom{c_2-2}2+c_2-3.
\]

This move yields a net gain of
\begin{align*}
 \k_t(G')-\k_t(G)
 &= \binom{c_1}{t} + \binom{c_2}{t-1} + \binom{c_2-2}{t} + \binom{c_2-3}{t-1} - \left(\binom{c_1}{t} + \binom{c_2}{t}\right)\\
 &= \binom{c_2}{t-1} + \left(\binom{c_2-2}t-\binom{c_2}t\right) + \binom{c_2-3}{t-1} \\
 &= \binom{c_2}{t-1} - \binom{c_2-1}{t-1} - \left(\binom{c_2-2}{t-1} - \binom{c_2-3}{t-1}\right) \\
 & = \binom{c_2-1}{t-2} - \binom{c_2-3}{t-2}\\
 & > 0.
\end{align*} 

In all cases except $b_1 = \binom{c_1}{2}$ and $b_2=1$ we have shown that there exists $\beta \ge 1$ such that $k_t(\cC(b_1)\cup\cC(b_2)) < k_t(\cC(b_1+\beta)\cup\cC(b_2-\beta))$ and also that $\cC(b_1+\beta)\cup\cC(b_2-\beta)\in \cG(b_1+b_2,r)$. We have that, except in the special case,

\begin{align*}
k_t(\cC(b_1)\cup\cC(b_2)) &< \begin{cases}
k_t(\cC(b_1+b_2)) & \text{if }b_1 + b_2 \le \binom{r+1}{2}\\
k_t(\cC(\binom{r+1}{2})\cup\cC(b_1+b_2-\binom{r+1}{2})) &\text{otherwise}
\end{cases}\\
&= g_t(b_1+b_2,r).\qedhere
\end{align*}
\end{proof}

\begin{cor}\label{connected}
For $t \ge 3$, if Conjecture \ref{conj:allt} holds for numbers of edges up through $m-1$, and $G \in \cG(m,r)$ is not connected, then $k_t(G) \le g_t(m,r)$. 
\end{cor}

\begin{proof}
Suppose $G = N \cup M$, where $0 \ne e(N) = a_1\binom{r+1}{2}+b_1$, $0 \ne e(M) = a_2\binom{r+1}{2}+b_2$, and $m = e(G) = e(M)+e(N) = a\binom{r+1}{2}+b$ with $0 \le b_1, b_2, b \le \binom{r+1}{2}-1$. By Conjecture \ref{conj:allt} for smaller values of $m$,\begin{align*}
k_t(G) &= k_t(N) + k_t(M)\\
& \le k_t(a_1K_{r+1}\cup \cC(b_1)) + k_t(a_2K_{r+1}\cup \cC(b_2))\\
&= k_t((a_1+a_2)K_{r+1})+k_t(\cC(b_1)\cup \cC(b_2)).\end{align*}

If $b_1, b_2 \le 1$, then $a = a_1+a_2$, $b = b_1 + b_2 \le 2$, and $k_t(\cC(b_1)\cup \cC(b_2)) = 0 = k_t(\cC(b_1+b_2))$, so we have shown $k_t(G) \le k_t(aK_{r+1}\cup\cC(b))= g_t(m,r)$.

If $b_1 = \binom{c_1}{2}$ and $b_2=1$, then $a=a_1+a_2$, $b=b_1+b_2$, and $k_t(\cC(b_1)\cup \cC(b_2))= k_t(\cC(b_1+b_2))$, so we have shown $k_t(G) \le k_t(aK_{r+1}\cup\cC(b))=g_t(m,r)$.

In all other cases, by Lemma \ref{lem:b1b2} we have $k_t(\cC(b_1)\cup \cC(b_2)) < g_t(b_1+b_2,r)$, so
\begin{align*}
k_t(G) &\le k_t((a_1+a_2)K_{r+1})+k_t(\cC(b_1)\cup \cC(b_2)) \text{ [shown above]}\\
&< k_t((a_1+a_2)K_{r+1})+g_t(b_1+b_2,r)\\
&= g_t\parens[\big]{(a_1+a_2)\binom{r+1}{2} + b_1+b_2,r}\\
&= g_t(m,r).\qedhere\end{align*}
\end{proof}

\begin{cor}\label{Delta=r} 
    If $t \ge 3$, $r\ge 2$, $m \ge \binom{r+1}{2}+1$, $G \in \cG(m,r)$ has $k_t(G) = f_t(m,r)$, and Conjecture \ref{conj:allt} holds for maximum degree at most $r-1$ and for numbers of edges up through $m-1$, then $\Delta(G) = r$.
\end{cor}

\begin{proof}
The statement is trivial for $r=2$, so assume $r \ge 3$. Suppose $\Delta(G) \le r-1$, so $G \in \cG(m,r-1)$. By Conjecture \ref{conj:allt} for $r-1$, we have $k_t(G) \le k_t(aK_r\cup \cC(b))$ for $m = a\binom{r}{2} + b$ and $0 \le b < \binom{r}{2}$.

If $b \ge 2$, then $k_t(K_r\cup\cC(b)) < f_t(\binom{r}{2}+b,r)$ by Lemma \ref{lem:b1b2}. Therefore \begin{align*}
k_t(aK_r\cup \cC(b)) &= k_t((a-1)K_r) + k_t(K_r\cup\cC(b))\\
&< k_t((a-1)K_r) + f_t\parens[\big]{\binom{r}{2}+b,r}\\
&\le f_t\parens[\big]{(a-1)\binom{r}{2},r} + f_t\parens[\big]{\binom{r}{2}+b,r}\\
&\le f_t\parens[\big]{(a-1)\binom{r}{2} + \binom{r}{2}+b,r} = f_t(m,r)\end{align*} since $f_t$ is a superadditive function of $m$.

Otherwise, $b \le 1$, and $m \ge \binom{r+1}{2}+1$, so $a \ge 2$. Notice that $k_t(K_r\cup \cC(b)) = k_t(\cC(r,b))$.
\begin{align*}
k_t(aK_r\cup \cC(b)) &= k_t((a-2)K_r) + k_t(K_r) + k_t(K_r\cup \cC(b))\\
&= k_t((a-2)K_r) + k_t\parens[\big]{\cC(r,0)} + k_t(\cC(r,b))\\
&= k_t((a-2)K_r) + k_t(\cC(r,0)\cup\cC(r,b))\\
&< k_t((a-2)K_r) + f_t\parens[\Big]{2\binom{r}{2}+b, r}\text{ [by Lemma \ref{lem:b1b2} since $r \ge 3$]}\\
&\le f_t\parens[\Big]{(a-2)\binom{r}{2}, r} + f_t\parens[\Big]{2\binom{r}{2}+b, r}\\
&\le f_t\parens[\Big]{a\,\binom{r}{2}+b,r} = f_t(m,r),
\end{align*}
since $f_t$ is a superadditive function of $m$.
Therefore any $G \in \cG(m,r-1)$ is suboptimal, and $k_t(G) = f_t(m,r)$ implies $\Delta(G) = r$.
\end{proof}

\section{Edge Weights, Clusters, and Folding}\label{sec:clusters}

In this section we introduce a `folding' operation that acts on a graph containing a large subset of vertices with as many common neighbors as is possible. To this end we define the \emph{weight} of a pair of vertices to be the number of common neighbors they have: $w(xy) = |N(x)\cap N(y)|$ for any $x, y \in V(G)$. In particular if $xy$ is an edge this is the number of triangles containing that edge. If $xy$ is a non-edge this is the number of triangles we would gain by adding that pair as an edge. The maximum possible weight of an edge $xy$ is $w(xy)=r-1$, which occurs exactly when $d(x)=d(y)=r$ and $N[x]=N[y]$. (Note that a non-edge can have weight $r$. If $x\not\sim y$, $d(x) = d(y) = r$, and $N(x) = N(y)$, then $x$ and $y$ have $r$ common neighbors, but we cannot add the edge $xy$ to complete $r$ triangles as $x$ and $y$ already have the maximum degree.)

\begin{definition}
	An edge $xy\in E(G)$ is called \emph{tight} if $w(xy)=r-1$.  A complete subgraph in $G$, all of whose edges are tight, is called a \emph{tight clique}, and a maximal tight clique is called a \emph{cluster}.
\end{definition}

For any cluster $T$, let $S_T = \bigcap_{v \in T} N(v)$, the set of common neighbors of $T$. For $v \in T$, every other vertex $x \in T$ has $r-1$ neighbors in common with $v$. They must be the same common neighbors for each $x$ since $d(v) \le r$, so $|T\cup S_T| = r+1$. For any $v \in T$, the closed neighborhood of $v$ is $N[v] = T\cup S_T$.

If $G$ contains a cluster of size $r+1$, then it contains a copy of $K_{r+1}$, and it is disconnected. 
What we'll investigate is the situation in which there is a cluster $T$ with $\abs{T}<r+1$, and try to understand the edges missing from $S_T$. Let $R_T = \overline{G[S_T]}$. Since $T$ is maximal, $\delta(R_T) \ge 1$. The vertices in $S_T$ may have neighbors in $T$, $S_T$, and $V(G)\setminus(T\cup S_T)$. Let $B_T$ be the graph of edges $uv$ such that $u \in S_T$ and $v \in V(G)\setminus(T\cup S_T)$.

We will consistently write $t$ for $|T|$, $S$ for $S_T$, $s$ for $|S_T|$, $R$ for $R_T$, and $B$ for $B_T$. We refer to $R$ as the \emph{red graph} and the edges of $B$ as \emph{blue edges}. We will also define $a$, $b$, $c$, and $d$ from $m$ and $r$ by $m = a\binom{r+1}{2}+b$, $0 \le b < \binom{r+1}{2}$, and $b = [c,d]$. 

The following simple bound on the weight of a blue edge will aid in determining effects of local moves.

\begin{lemma}\label{s-2} Each blue edge has weight at most $s-2$.\end{lemma}
\begin{proof}
Let $xy$ be a blue edge with $x \in S$. 
We will separately count the possible common neighbors $z$ in and out of $S$. There are at most $s - 1 - d_R(x)$ neighbors $z$ of $x$ that are in $S$. For $z \notin S$, $xz$ is also a blue edge, and $d_B(x) \le d_R(x)$ to maintain $d_G(x) \le r$. The maximum number of such $z$'s then is $d_R(x) - 1$. In total, $xy$ is in at most $s-1-d_R(x) + d_R(x) - 1 = s-2$ triangles.
\end{proof}

We will often find it useful to delete all the blue edges from a cluster and add all the red edges. We call this operation \emph{folding}:

\begin{definition}
	For $G \in \cG(m,r)$ with a cluster $T$ and $e(B) \ge e(R)$, we define a new graph $G_T \in \cG(m,r)$ by converting $T\cup S_T$ into a complete subgraph (of size $r+1$) and deleting all the edges in $B_T$. In other words, we define the \emph{folding of $G$ at $T$} by
	\[
		G_T = G + \binom{S_T}{2} - E(B_T).
	\]
\end{definition}

The graph $G_T$ contains a $K_{r+1}$, has maximum degree at most $r$, and has at most $m$ edges since $e(B) \ge e(R)$. If we can show that $e(B)\ge e(R)$ and $\k_3(G_T)\ge \k_3(G)$, then by induction on $m$,
\[
	\k_3(G) \le \k_3(G_T) \le g_3(e(G_T),r) \le g_3(m,r).
\]

\section{Excluded Red Graphs $R$}\label{sec:red}

In this section we identify several graphs that cannot occur as $R$ in an extremal graph $G$ because folding (when $e(B) \ge e(R)$) or another local move (when $e(B) < e(R)$) would increase the number of triangles. Our first step toward identifying when folding increases the number of triangles in $G$ will be to give an upper bound on the number of \emph{blue triangles}, or triangles containing two blue edges. We will use a compression argument, in which we determine which configuration of blue edges is least helpful to us.

\begin{definition}For vertices $x\not\sim_G y$, the \emph{compression of $G$ from $x$ to $y$}, denoted $G_{x\to y}$, is the graph obtained from $G$ by deleting all edges between $x$ and $N(x)\setminus N(y)$ and adding all edges from $y$ to $N(x)\setminus N(y)$.\end{definition}

We define an auxiliary function to use in the compression argument.

\begin{definition}For a graph $G$, let $d_2(G) := \sum_{v \in G} (d(v))^2$.\end{definition}

We will use compressions to maximize the following function and bound the number of blue triangles.

\begin{definition}For a graph $H$ and a bipartite graph $B$ with bipartition $(V(H),Y)$, we define \[\psi_H(B)=\sum_{v\in V(H)}\binom{d_B(v)}{2} + \sum_{v\in Y}|\{i,j\in V(H): i \ne j, iv, jv \in E(B), ij \notin E(H)\}|.\]\end{definition}

This function counts the number of blue triangles when applied to the situation where $H$ is the red graph $R_T$ and $\parens[\big]{\bigcup_{e\in E(B_T)} e}\setminus S_T$ induces a complete subgraph. Thus it serves as an upper bound on the number of blue triangles for a given red graph.

\begin{lemma}\label{lem:compincr}For a graph $H$, a bipartite graph $B$ with bipartition $(V(H), Y)$, and vertices $x, y \in Y$ with $N(x) \not\subseteq N(y)$ and $N(y) \not\subseteq N(x)$, \[\psi_H(B) \le \psi_H(B_{x\to y}) \text{ and } d_2(B) < d_2(B_{x \to y}).\]\end{lemma}

\begin{proof}For all $v \in V(H)$, we have $d_B(v)=d_{B_{x\to y}}(v)$ and therefore $\sum_{v\in V(H)}\binom{d_{B_{x\to y}}(v)}{2}=\sum_{v\in V(H)}\binom{d_B(v)}{2}$. $H$ is fixed, and $x$ is the only vertex of $Y$ that loses neighbors, so any decrease in $\sum_{v\in Y}|\{i,j\in V(H): i\ne j, iv, jv \in E(B), xy \notin E(H)\}|$ would be from the $x$ term: pairs $i,j \in H$ that are neighbors of $x$ and not adjacent in $H$. If $i, j \in N(y)$, then the edges $ix$ and $jx$ remain after the compression, so the pair $i, j$ is still counted in the $x$ term of the summation. If one or both of $i, j \notin N(y)$, then the pair $i, j$ is counted in the $y$ term of the summation after the compression but not before, compensating for the loss in the $x$ term.

Let $\ell := |N(x) \setminus N(y)| > 0$. Then \[d_2(B_{x\to y}) - d_2(B) = (d(x) - \ell)^2 + (d(y)+\ell)^2 - d(x)^2 - d(y)^2 = 2\ell(\ell+d(y)-d(x)) > 0.\]\end{proof}

Bipartite threshold graphs can be defined in different ways, but the following is the one we will use.

\begin{definition}[See \cite{Hammer}]\label{def:threshold}
	A graph $G$ is a \emph{bipartite threshold graph} if and only if $G$ is bipartite and the neighborhoods of vertices in one of the partite sets are linearly ordered by inclusion.
\end{definition}

\begin{lemma}\label{lem:family} Suppose that $\mathcal{B}$ is a family of bipartite graphs on a fixed vertex set $(X,Y)$ such that for any $B' \in \mathcal{B}$ and $x, y \in Y$ with $N_{B'}(x) \not\subseteq N_{B'}(y)$ and $N_{B'}(y) \not\subseteq N_{B'}(x)$, we also have $B'_{x\to y}\in \mathcal{B}$. If $B \in \mathcal{B}$ and $d_2(B) = \max\{d_2(B') : B' \in \mathcal{B}\}$, then $B$ is a bipartite threshold graph.\end{lemma}

\begin{proof}
In $B$, suppose there are $x, y \in Y$, $N_B(x) \not\subseteq N_B(y)$, and $N_B(y) \not\subseteq N_B(x)$. Then $d_2(B) < d_2(B_{x\to y})$ by Lemma \ref{lem:compincr}, but $d_2(B) \ge d_2(B_{x\to y})$ because $B_{x\to y} \in \mathcal{B}$ and $d_2(B) = \max\{d_2(B') : B' \in \mathcal{B}\}$. Therefore every pair $x, y \in Y$ has $N(x) \subseteq N(y)$ or $N(x) \supseteq N(y)$. By Definition \ref{def:threshold}, $B$ is a bipartite threshold graph.
\end{proof}

\begin{lemma}\label{lem:threshold} Given a graph $H$, among bipartite graphs $B$ with bipartition $(V(H), Y)$ and a fixed number of edges, some bipartite threshold graph $B$ maximizes $\psi_H(B)$.\end{lemma}

\begin{proof}
Consider the family $\mathcal{B}$ of bipartite graphs on $(V(H),Y)$ with the specified number of edges that maximize $\psi_H(B)$. It is closed under compressions by Lemma \ref{lem:compincr}. By Lemma \ref{lem:family}, some $B \in \mathcal{B}$ is a bipartite threshold graph.
\end{proof}

\begin{cor}\label{cor:bluetriangles}For a given red graph $R$, the number of blue triangles at $R$ is at most \[\sum_{v\in R} \binom{d_R(v)}{2} + \frac12 \sum_{v\in R} d_R(v)(s-1-d_R(v)).\]
\end{cor}

\begin{proof}
Let $Y = V(B)\setminus S$. There are two types of blue triangles. \begin{enumerate}
\item One vertex in $R$ and two blue edges at that vertex: There are at most $\sum_{v\in R} \binom{d_B(v)}{2}$ triangles of this type.
\item Two adjacent vertices in $S$ (non-adjacent in $R$) and one blue edge at each of these vertices to the same third vertex outside $S$: There are 
\[
	\sum_{\substack{i,j\in R\\ i\not\sim_R j\\i\ne j}} |\{v \in Y : iv, jv \in E(B)\}| = \sum_{v\in Y}|\{i,j\in R: i \ne j, iv, jv \in E(B), ij \notin E(R)\}|
\]triangles of this type.
\end{enumerate}

The total number of blue triangles at $R$ is exactly $\psi_R(B)$, which by Lemma \ref{lem:threshold} is maximized by some bipartite threshold blue graph $B$. The definition of $R$ and $\Delta(G) \le r$ imply $d_B(v) \le d_R(v)$ for all $v \in R$. Adding edges can only increase the number of blue triangles, so we may assume that $d_B(v) = d_R(v)$ for all $v \in R$. These blue degrees in $R$ completely determine the graph $B$ because the neighborhoods of vertices in $Y$ are nested. Any vertex $i \in R$ has $d_R(i)$ neighbors in $Y$, and a pair $i,j \in R$ has $\min(d_R(i),d_R(j))$ common neighbors in $Y$.

\begin{align*}
\psi_R(B) &= \sum_{v\in R} \binom{d_B(v)}{2} + \sum_{\substack{i,j\in R\\ i\not\sim_R j\\i\ne j}} |\{v \in Y : i \ne j, iv, jv \in E(B)\}|\\
&\le \sum_{v\in R} \binom{d_R(v)}{2} + \sum_{\substack{i,j\in R\\ i\not\sim_R j\\i\ne j}} \min(d_R(i),d_R(j))\\
&\le \sum_{v\in R} \binom{d_R(v)}{2} + \frac{1}{2}\sum_{\substack{i,j\in R\\ i\not\sim_R j\\i\ne j}} (d_R(i)+d_R(j))\\
&= \sum_{v\in R} \binom{d_R(v)}{2} + \frac{1}{2}\sum_{v\in R}d_R(v)(s-1-d_R(v)).\qedhere\end{align*}
\end{proof}

Having given an upper bound on the number of blue triangles, we turn our attention to a lower bound on the number of red triangles gained from folding $G$ at $T$.

\begin{definition}For a graph $R$ with $s$ vertices, we define \[Q(R) = (r+1 -s )e(R) + k_3(R) - \sum_{v\in R} \binom{d_R(v)}{2}.\]\end{definition}

\begin{lemma}\label{lem:QR}If $G$ has a cluster $T$ with $e(B) \ge e(R)$, then $k_3(G_T) - k_3(G) \ge Q(R).$\end{lemma}

\begin{proof}
We will bound $k_3(G_T) - k_3(G)$ by counting the triangles gained and lost by folding at $T$. Three types of triangles are gained by folding.\begin{enumerate}
\item One red edge and one vertex in $T$: There are $te(R)$ triangles of this type.
\item One or two red edges, and all vertices in $S$: Each triangle of this type contains exactly two unordered pairs of an incident edge and non-edge, and each such pair occurs in a triangle of this type. By counting these pairs and dividing by two, we find there are $\frac12 \sum_{v\in R}d_R(v)(s-1-d_R(v))$ triangles of this type.
\item Three red edges: There are $k_3(R)$ triangles of this type.
\end{enumerate}

Thus the total number of triangles gained from folding is 
\[
    te(R)+\frac12 \sum_{v\in S}d_R(v)(s-1-d_R(v))+k_3(R).
\]
The total number of triangles lost from folding is at most 
\[
    \sum_{v\in R} \binom{d_R(v)}2 + \frac12\sum_{v\in R} d_R(v)(s-1-d_R(v))
\]
by Corollary \ref{cor:bluetriangles}.

The net gain from folding, $k_3(G_T) - k_3(G)$, is at least
\begin{align*}
    te(R)+\frac12 \sum_{v\in S}d_R(v)&(s-1-d_R(v))+k_3(R)
     - \sum_{v\in S} \binom{d_R(v)}{2} - \frac12 \sum_{v\in S} d_R(v)(s-1-d_R(v))\\
    &= te(R)+k_3(R) - \sum_{v\in S} \binom{d_R(v)}{2}
     = Q(R).\qedhere
\end{align*}
\end{proof}

\begin{lemma}\label{lem:half}If $s \le \frac{r+2}{2}$, then $Q(R) \ge 0$, with equality if and only if $R = E_s$.
\end{lemma}

\begin{proof}
Note $Q(E_s) = 0$. We will show that $E_s$ is the unique minimizer of $Q(R)$. If $R$ has an edge $xy$, then deleting it would strictly decrease $Q(R)$. We will consider the net change in $Q(R)$ term by term. The change in the $(r+1-s)e(R)$ term is $-(r+1-s)$ since we are losing one edge. The change in the $k_3(R)$ term is $-|N_R(x)\cap N_R(y)|$. The change in the final term is $d_R(x) + d_r(y) - 2$. Therefore
\begin{align*}
Q(R-xy) - Q(R) &= -(r+1-s) - |N_R(x)\cap N_R(y)| + d_R(x) + d_R(y) - 2\\
&= -r-3+s - |N_R(x)\cap N_R(y)| + |N_R(x)| + |N_R(y)|\\
&= -r-3+s + |N_R(x) \cup N_R(y)|\\
&\le -r -3+2s\\
&\le -r -3 + r+2 = -1 < 0.\qedhere
\end{align*}
\end{proof}

\begin{thm}\label{half} If $G \in \cG_C(m,r)$ has a cluster $T$ with $s \le \frac{r+2}{2}$, then $k_3(G) < f_3(m,r)$.\end{thm}

\begin{proof}

As $G$ is connected and $m \ge \binom{r+1}{2}$, there is at least one blue edge.

\begin{case}$0 < e(B) < e(R)$\end{case}

We will delete all of the blue edges and add $e(B)$ of the red edges, which will maintain the number of edges and the bound on the maximum degree. The loss from deleting the blue edges is at most $(s-2)e(B)$ by Lemma \ref{s-2}. The gain from the red edges is at least $te(B)$, as each red edge has $t$ common neighbors in $T$. Therefore the net gain is at least\begin{align*}
te(B) - (s-2)e(B) &= (t-s+2)e(B)\\
&=(r+1-2s+2)e(B) = (r+3 - 2s)e(B)\\
&\ge (r+3- (r+2))e(B)\\
&= e(B) > 0,\end{align*}
so the graph was not extremal.

\begin{case}$e(B) \ge e(R)$\end{case}

By Lemmas \ref{lem:QR} and \ref{lem:half}, since $\delta(R) \ge 1$, folding $G$ at $T$ strictly increases the number of triangles, so $G$ is not extremal.
\end{proof}

\begin{lemma}\label{lem:matching}If $G \in \cG_C(m,r)$ has a cluster $T$ with $\Delta(R) \le 1$, then $k_3(G) < f_3(m,r)$.\end{lemma}

\begin{proof}Any red graph has $\delta(R) \ge 1$, so $R = \frac{s}{2}K_2$.

\begin{case}$0 < e(B) < e(R)$\end{case}

There is a red edge that is not incident to any blue edges. Add that red edge, and delete a blue edge. The blue edge is in at most $s-2$ triangles by Lemma \ref{s-2}. The red edge has weight $r-1$. The net gain in number of triangles is at least $r-1 -(s-2) = r+1-s \ge 1$, so $k_3(G) < f_3(m,r)$.

\begin{case}$e(R) \le e(B)$\end{case}

By Lemma \ref{lem:QR}, $k_3(G_T) - k_3(G) \ge Q(\frac{s}{2}K_2) = (r+1-s)(\frac{s}{2}) + 0 - 0 = st/2 > 0$, so $k_3(G) < f_3(m,r)$.
\end{proof}

\begin{lemma}\label{lem:D2}If $\Delta(R) \le 2$, then $Q(R) \ge 0$, with equality only when $t=1$ and $R$ is a disjoint union of non-triangle cycles.\end{lemma}

\begin{proof}Any red graph has $\delta(R) \ge 1$. Let $k$ be the number of vertices with red degree 2, so $s-k$ is the number of vertices with red degree 1. By the degree sum formula, $2e(R) = 2k + (s-k) = s+k$.

\begin{align*}
Q(R) &= (r+1-s)e(R) + k_3(R) - \sum_{v\in R}\binom{d_R(v)}{2}\\
&\ge t\frac{s+k}{2} + k_3(R) - k\\
&= st/2 + k(t/2-1) + k_3(R).
\end{align*}
For $t \ge 2$, each term of this last expression is non-negative, and $st/2 >0$, so $Q(R) > 0$. For $t = 1$, 
\[
st/2 + k(t/2-1) +k_3(R)= s/2 -k/2+k_3(R) = (s-k)/2+k_3(R) \ge 0,
\]
with equality only when all vertices of $R$ have red degree 2 and there are no triangles, i.e. $R$ is a disjoint union of non-triangle cycles.
\end{proof}

\begin{lemma}\label{lem:D2b}If $G \in \cG_C(m,r)$ has a cluster $T$ with $\Delta(R) \le 2$ and $t \ge 2$, then $k_3(G) < f_3(m,r)$.\end{lemma}

\begin{proof}
There is at least one blue edge because $G$ is connected and $m \ge \binom{r+1}{2}$. By Lemma \ref{lem:D2}, $Q(R) > 0$, so if $e(B) \ge e(R)$, then $k_3(G) < f_3(m,r)$. Now suppose $0 < e(B) < e(R)$.

We will show that there is a red edge that is incident to at most one blue edge. The average number of blue edges at a given red edge is
\begin{align*}
\frac{1}{e(R)}\sum_{f \in E(R)} |\{g \in E(B) : f\cap g \neq \emptyset\}| 
&= \frac{1}{e(R)}\sum_{g \in E(B)} |\{f \in E(R) : f\cap g \neq \emptyset\}|\\
&\le \frac{1}{e(R)}\sum_{g \in E(B)} 2\\
&= \frac{2e(B)}{e(R)} <2.
\end{align*}
Thus there is a red edge $xy$ incident to at most one blue edge. The vertices $x$ and $y$ may each have one other neighbor in $R$. The remaining $r+1-2-2 = r-3$ vertices in the cluster are neighbors of both $x$ and $y$ in $G$, so $xy$ has weight at least $r-3$. The blue edge has weight at most $s-2$ by Lemma \ref{s-2}. Deleting the blue edge (or any blue edge, if there is none at $xy$) and adding the red edge $xy$ yields a net gain of at least $r-3 - (s-2) = r-s -1 = t-2 \ge 0$ triangles since $t \ge 2$, a weak increase that reduces the number of red edges. By induction on the number of red edges, we are done. (The base case $e(R) = 1$ is done by Lemma \ref{lem:matching}.)
\end{proof}

\begin{thm}\label{thm:s=2}If $G \in \cG_C(m,r)$ for $r \ge 3$ has a cluster with $e(R) \in \{1,2\}$, then $k_3(G) < f_3(m,r)$.\end{thm}

\begin{proof}
If $e(R) = 1$, then $R = K_2$, and by Lemma \ref{lem:matching} we're done. 
If $e(R) = 2$, then $R = 2K_2$ or $R = P_3$. For $R = 2K_2$, we are again done by Lemma \ref{lem:matching}.

Consider $R = P_3$. $G$ is connected and $m \ge \binom{r+1}{2}$, so $e(B) \ge 1$. Any blue edges have weight at most $s-2=1$ by Lemma \ref{s-2}. Each red edge has weight $r-2$. If $e(B) = 1$, delete the blue edge and add a red edge for a net gain of at least $r-2 - 1 = r-3$ triangles, a weak increase for $r \ge 3$. This reduces to the case $R = K_2$.

Otherwise, $e(B) \ge 2 = e(R)$, so we may fold. $Q(P_3) = (r+1-3)2+0-1 = 2r-5 \ge 1$ for $r \ge 3$, so $k_3(G) < f_3(m,r)$ by Lemma \ref{lem:QR}.
\end{proof}

\begin{thm}\label{thm:e=3}If $G \in \cG_C(m,r)$ for $r \ge 7$ has a cluster with $e(R)=3$, then $k_3(G) < f_3(m,r)$.\end{thm}

\begin{proof}
Suppose $G$ is extremal. By Theorem \ref{half}, we may assume $s \ge \frac{r+3}{2} \ge 5$. There are two graphs $R$ with 3 edges, at least 5 vertices, and $\delta(R) \ge 1$: $R = 3K_2$ and $R = P_3\cup K_2$. Both have $\Delta(R) \le 2$ and $t \ge 2$, so by Lemma \ref{lem:D2b}, $k_3(G) < f_3(m,r)$.
\end{proof}

\begin{thm}\label{thm:e=4}If $G \in \cG_C(m,r)$ for $r \ge 8$ has a cluster with $e(R)=4$, then $k_3(G) < f_3(m,r)$.\end{thm}

\begin{proof}
Suppose $G$ is extremal. By Theorem \ref{half}, we may assume $s \ge \frac{r+3}{2} \ge 5.5$ so $s \ge 6$.

There are five graphs $R$ with 4 edges, $\delta(R) \ge 1$, and $s \ge 6$ vertices. The ones on 6 vertices are $K_{1,3}\cup K_2$, $2P_3$, and $P_4\cup K_2$. On 7 vertices, $R$ can only be $2K_2\cup P_3$, and on 8 vertices, $R$ can only be $4K_2$. All of these except $K_{1,3}\cup K_2$ have either $\Delta(R) \le 1$ or $\Delta(R) \le 2$ and $t \ge 2$, so Lemmas \ref{lem:matching} and \ref{lem:D2b} show $G$ is not extremal.

Suppose $R = K_{1,3}\cup K_2$. Then $Q(R) = (r+1-6)4+0-3 = 4r-23 > 0$. If $e(B) \ge e(R)$, then folding will increase the number of triangles, and $G$ was not extremal. Otherwise, $1 \le e(B) \le 3$. The isolated $K_2$ in $R$ has weight $r-1 \ge 7$. Any blue edge has weight at most $s-2 = 4$. If there is only one blue edge incident to the red $K_2$ (or none), then delete it (or any blue edge) and add the edge corresponding to the red $K_2$, which increases the number of triangles by at least $7-4=3$. Otherwise, there are 2 blue edges incident to the red $K_2$, so 0 or 1 blue edges incident to the red $K_{1,3}$. Deleting a blue edge (from the $K_{1,3}$ if needed) and adding one of the red $K_{1,3}$ edges increases the number of triangles by at least $5-4 = 1$. Therefore $G$ was not extremal.
\end{proof}

\section{Degree Multiset Optimization}\label{sec:seq}

The number of triangles in a graph can be bounded easily (but crudely) as follows: 
\begin{equation}\label{eq:k3}
	k_3(G) = \frac{1}{3}\sum_{v\in V(G)}k_3(v) = \frac13\sum_{v\in V(G)} e(N(v)) \le \frac13\sum_{v\in V(G)}\binom{d(v)}2.
\end{equation}	

We have shown, in Section \ref{sec:red}, that vertices of degree $r$ lie in clusters with certain red graphs forbidden.
In this section we consider the possible degree multisets to give an upper bound on the number of triangles in graphs $G$ that do not have any of the excluded red graphs $R$ considered in Section \ref{sec:red}. We consider the constraints imposed by the trivial bound above.

For instance, when $r \ge 3$, Theorem \ref{thm:s=2} shows that in an extremal connected graph, every cluster is missing at least three edges. Similarly, for $r=7$ and $8$, Theorems \ref{thm:e=3} and \ref{thm:e=4} show that in an extremal connected graph, every cluster is missing at least four or five edges, respectively. We define upper bounds for $\k_3(G)$ based on the degree multiset of $G$.

\begin{definition} 
	For given $k, m, r$ and for $d \in \{0\}\cup[r]$, let 
	\[
		w(d) = \begin{cases}\binom{d}{2} & \text{ if } d \ne r\\
						    \binom{r}{2} - k & \text{ if } d=r.
			   \end{cases}
	\]
Set 
\begin{align*}
	M_k(m,r) &= \frac{1}{3}\max \set[\Big]{ \sum_{d \in D} w(d) : \sum_{d \in D} d = 2m}\text{ and}\\
	M^*_k(m,r) &= \frac{1}{3}\max \set[\Big]{ \sum_{d \in D} w(d) : \sum_{d \in D} d = 2m,~r \in D},
\end{align*} where in both cases $D$ is a multiset of arbitrary size from $\{0\}\cup[r]$.
\end{definition}

\begin{lemma}\label{lem:ub}
	For any $G\in \cG(m,r)$ such that every cluster has $e(R)\ge k$ we have $\k_3(G)\le M_k(m,r)$. If in addition the Main Theorem holds for maximum degree at most $r-1$ and for numbers of edges up through $m-1$, and $m \ge \binom{r+1}2+1$, then we have $k_3(G) \le \floor{M_k^*(m,r)}$.
\end{lemma}

\begin{proof}
Apply (\ref{eq:k3}), and observe that (in each case) the degree multiset of $G$ is one of the candidates in the maximization by Corollary \ref{Delta=r}.
\end{proof}

We will show in many cases the lower bounds we have on the number of edges in red graphs ensure that the upper bound from Lemma~\ref{lem:ub} is less than the number of triangles in $aK_{r+1}\cup\cC(b)$, proving the Main Theorem in those cases.
We say that a multiset $D$ is \emph{optimal} if it achieves the maximum $M_k(m,r)$ and has no $0$ entries.

\begin{lemma}\label{lem:r-2}
	An optimal multiset $D$ contains at most one entry from $[r-2]$.
\end{lemma}

\begin{proof}
	If an optimal multiset $D$ had two elements in $[r-2]$, $x \le y$, then changing these two elements to $x-1$ and $y+1$ maintains the sum of $2m$ while increasing the weighted sum $\sum_{d \in D} w(d)$, contradicting the optimality of $D$. If $x-1 = 0$, then the 0 is discarded, and $x=1$ and $y$ are replaced by $y+1$. The weighted sum is increased because 
\begin{align*}
w(x-1) + w(y+1) - (w(x) + w(y)) &= \binom{x-1}{2} + \binom{y+1}{2} - \binom{x}{2} - \binom{y}{2}\\
&= y - (x-1) = y - x + 1 \ge 1.\qedhere
\end{align*}\end{proof}

\begin{lemma}\label{lem:r-1-d}For any $m$ and $r \le 2k+1$, if an optimal multiset $D$ contains an entry $d \in [r-2]$ and at least $r-1-d$ copies of $r$, then there is an optimal multiset containing no entries from $[r-2]$.\end{lemma}

\begin{proof}
We will change the $r-1-d$ copies of $r$ and the $d \in [r-2]$ to $r-d$ copies of $r-1$. The resulting multiset contains no entries from $[r-2]$ by Lemma \ref{lem:r-2}.

This maintains the sum of $D$ since $d + (r-1-d)r = d + r^2 - r - dr = (r-d)(r-1)$. The increase in the weighted sum is given by
\begin{align*}
(r-d)\binom{r-1}{2} &- \left( \binom{d}{2} + (r-1-d)\left(\binom{r}{2} - k\right)\right)\\
					&= -\frac12\left((r-d)(r-1)(2)+d^2-d-r^2+r-2kr+2dk+2k\right)\\
					&= -\frac12\left(d^2 - 2dr + (1+2k)d + r^2 - (2k+1)r + 2k\right)\\
					&= -\frac12\left(d -(r-1)\right)\left(d-(r-2k)\right)\\
					&\ge 0,
\end{align*} for $r-2k \le d \le r-1$. Since $r \le 2k+1$, we have $r-2k \le 1 \le d \le r-2 < r-1$, and we weakly increased the weighted sum.\end{proof}

\begin{lemma}\label{lem:rr-1}If $D$ is an optimal multiset containing only $r-1$'s and $r$'s with $k\ge r/2$, then there is an optimal multiset containing at most $r-2$ copies of $r$, and the rest $r-1$'s.\end{lemma}

\begin{proof}
Changing $r-1$ copies of $r$ to $r$ copies of $r-1$ maintains the sum of $D$. 

The weighted sum increases by
\begin{align*}
r\binom{r-1}{2} - (r-1)\left(\binom{r}{2}-k\right) &= (r-1)\left(\frac{1}{2}r(r-2)-\binom{r}{2}+ k\right)\\
&= (r-1)\left(k-r/2\right) \ge 0.
\end{align*} Change any $r-1$ copies of $r$ to $r$ copies of $r-1$ repeatedly to get an optimal multiset with at most $r-2$ copies of $r$.
\end{proof}

The next few results consider the case where $k=\ceil{r/2}$.

\begin{lemma}\label{lem:nod}
	If $m \ge \binom{r+1}{2}+1$ and $k=\ceil{r/2}$ then there is an optimal multiset containing no elements of $[r-2]$.
\end{lemma}

\begin{proof}
	Suppose $D$ is an optimal multiset containing an entry $d \in [r-2]$.

	\begin{case}$D$ contains at least $r-1-d$ copies of $r$.\end{case}

	Lemma \ref{lem:r-1-d} shows there is an optimal multiset containing no entries from $[r-2]$.

	\begin{case}$D$ contains at least $d$ copies of $r-1$.\end{case}

	We will change the $d$ copies of $r-1$ and the $d \in [r-2]$ to $d$ copies of $r$. The resulting multiset contains no entries from $[r-2]$ by Lemma \ref{lem:r-2}. This maintains the sum: $d + d(r-1) = dr$.	It also increases the weighted sum: $d\left(\binom{r}{2} - k\right) - \left(\binom{d}{2} + d\binom{r-1}{2}\right) = -\frac{d}{2}(d - 2r + (2k+1))$, which is positive between $d=0$ and $d = 2r-(2k+1)$. Since $k <r$, we have $2r-(2k+1) > 0 $.

	\begin{case}$D$ contains at most $r-2-d$ copies of $r$ and at most $d-1$ copies of $r-1$.\end{case}

	By Lemma \ref{lem:r-2}, $d$ is the only entry from $[r-2]$. Then $\sum_{d' \in D} d' \le d + (r-2-d)r + (d-1)(r-1) = d + r^2 - 2r - dr + dr - r - d + 1 = r^2 - 3r + 1$. We also have $2m \ge 2\left(\binom{r+1}{2} + 1\right) = (r+1)r + 2 = r^2 + r + 2$. Thus $r^2 + r + 2 \le 2m = \sum_{d \in D} d \le r^2 - 3r + 1$, which implies $4r+1 \le 0$, i.e. $r \le -\frac{1}{4}$, a contradiction.
\end{proof}

\begin{thm}\label{thm:seqoptsoln}If $m \ge \binom{r+1}{2}+1$ and $k=\ceil{r/2}$ then $3M_k(m,r) \le (r-2)m$.
\end{thm}

\begin{proof} By Lemma~\ref{lem:nod} we may assume an optimal multiset $D$ consists only of $x$ copies of $r$ and $y$ copies of $r-1$. By Lemma \ref{lem:rr-1}, we may further assume $0 \le x \le r-2$. The degree sum formula implies $2m = xr + y(r-1)$, so $y  = \frac{2m-xr}{r-1}$. 

The weighted sum of $D$ gives the value of $3M_k(m,r)$, which is
\begin{align*}
y\binom{r-1}{2} + x \left(\binom{r}{2}-k\right) &= \frac{2m-xr}{r-1}(r-1)(r-2)/2 - x\left(k - \frac{r(r-1)}{2}\right)\\
&= (r-2)m - \frac{x}{2}(2k-r)\\
&\le (r-2)m\end{align*} since $x \ge 0$ and $2k \ge r$.
\end{proof}

\begin{lemma}\label{lem:1to7}If $1 \le r \le 7$ and $a \ge 1$, then $(r-2)m < 3g_3(m,r)$.\end{lemma}

\begin{proof}
\begin{align*}
3k_3(aK_{r+1}\cup C(b)) &- (r-2)m\\
&= 3\left(a\binom{r+1}{3} + \binom{c}{3} + \binom{d}{2}\right) - (r-2)\left(a\binom{r+1}{2} + \binom{c}{2} + d\right)\\
&= a\binom{r+1}{2} + (c-r)\binom{c}{2} + d\left(\frac{3}{2}(d-1)-(r-2)\right)\\
&\ge \binom{r+1}{2} + (c-r)\binom{c}{2} + d\left(\frac{3}{2}(d-1)-(r-2)\right)\\
\end{align*}

Define $h_r(c) = (c-r)\binom{c}{2}$ and $q_r(d) = d(\frac{3}{2}(d-1)-(r-2))$. By taking derivatives, we find the minimum values for these functions are attained at $c = \frac{1}{3}(r+1 + \sqrt{r^2-r+1})$ and $d = \frac{2r-1}{6}$, with minimum values $q_r(d) \ge -\frac{1}{24}(2r-1)^2$ and $h_r(c) \ge -\frac{1}{54}\left[(2r-1)(r+1)(r-2) + 2(r^2-r+1)^{3/2}\right]$, resulting in \[3g_3(m) - (r-2)m \ge \binom{r+1}{2} -\frac{1}{54}\left[(2r-1)(r+1)(r-2) + 2(r^2-r+1)^{3/2}\right]-\frac{1}{24}(2r-1)^2.\] This function of $r$ is positive for $1 \le r \le 7$ with positive roots at $r \approx 0.14$ and $r \approx 7.21$.
\end{proof}

\begin{thm}\label{thm:seqopt} If $1 \le r \le 7$ and $m \ge \binom{r+1}{2}+1$, then $M_k(m,r) < f_3(m,r)$ for $k = \ceil{r/2}$.
\end{thm}

\begin{proof} Theorem \ref{thm:seqoptsoln} and Lemma \ref{lem:1to7} together show $3M_k(m,r) \le (r-2)m < 3g_3(m,r) \le 3f_3(m,r)$.
\end{proof}

\begin{lemma}\label{lem:r=8}
	For $r=8$, $m \ge \binom{r+1}{2}+1$, if the Main Theorem holds for $r \le 7$ and numbers of edges up through $m-1$, and every cluster has $e(R) \ge 5$, then $k_3(G) \le g_3(m,r)$.
\end{lemma}

\begin{proof}
We will restrict the multiset optimization problem for $r=8$, $k=5$ to multisets that contain at least one copy of $r=8$, since we know from Lemma \ref{Delta=r} that $G$ must have an $8$ in its degree multiset. Let $D$ be an optimal multiset that attains the maximum $M_5^*(m,8)$. It contains at most one element of $[6]$ by Lemma \ref{lem:r-2}.

First, $D$ contains at least one 8, so we may assume there is no 6: If there were both a 6 and an 8, then we could replace them with two 7's, increasing the weighted sum: $\binom{6}{2} + \binom{8}{2}-5 = 38$, and $2\binom{7}{2} = 42$. Therefore $D$ contains only 8's, 7's, and at most one element $d \in [5]$.

Second, if there are $d$ copies of $7$, then we change them and the $d \in [5]$ to $d$ copies of $8$.
This maintains the sum of $D$.
It also weakly increases the weighted sum: 
\[
	d\left(\binom{8}{2} - 5\right) - \left(\binom{d}{2} + d\binom{7}{2}\right) = -\frac{d}{2}(d - 5) \ge 0.
\]

Suppose $d \in [5]$ is in an optimal multiset $D$. Since $m \ge \binom{9}{2}+1$, either there are at least $7-d$ copies of $8$ in $D$, or there are at least $d$ copies of $7$ in $D$, but we have just handled the second case. In the first case, Lemma \ref{lem:r-1-d} gives the result. We conclude that there is an optimal multiset (among multisets containing an 8) that consists only of copies of 8 and copies of 7. By Lemma \ref{lem:rr-1}, changing seven 8's to eight 7's will increase $\sum_{d'\in D}w(d')$, so we have $x \in [7]$ copies of 7. By the degree sum formula, we have $y = \frac{2m-8x}{7}$ copies of 8.

The weighted sum of $D$ gives the value of $3M_5^*(m,8)$, which is\begin{align*}
\frac{2m-8x}{7}\binom{7}{2} + x\left(\binom{8}{2}-5\right) 6m-x < 6m.\end{align*} By the proof of Lemma \ref{lem:1to7}, $3f_3(m,8) - 6m \ge 36a + (c-8)\binom{c}{2} + d\left(\frac{3}{2}(d-1)-6\right)$.

The minimum value of $(c-8)\binom{c}{2}$ for integer values of $c$ between 0 and 8, inclusive, is $-30$, at $c=5,6$. For $c\ne 5,6$, the minimum value is $-24$.

The minimum value of $d\left(\frac{3}{2}(d-1)-6\right)$ for integer values of $d$ between 0 and 7, inclusive, is $-9$, at $d = 2,3$.  For $d \ne 2,3$, the minimum value is $-6$ at $d = 1,4$. For $d \ne 1,2,3,4$, the minimum value is less than $-6$.

For $a \ge 2$, $3f_3(m,8) - 6m \ge 72 - 30 - 9 > 0$. For $a = 1$ and $c \ne 5,6$, $3f_3(m,8) - 6m \ge 36 - 24 - 9 > 0$. For $a = 1$, $c = 5, 6$, and $d \ne 1, 2, 3, 4$, $3f_3(m,8) - 6m > 36 - 30 - 6 = 0$. 

For $a = 1$, $c=5,6$, and $d = 1, 2, 3, 4$, we calculate $3M_5^*(m,8) = 6m-x$ exactly and compare it to $3k_3(aK_{r+1}\cup\cC(b))$. We can determine $x$ for each $m$ because $2m = 8x + 7y \equiv x \mod 7$ and $x \in [7]$.
\begin{center}
	\begin{tabular}{llcc}
	$m$ & $x$ & $3M_5^*(m,8) = 6m-x$ & $3g_3(m,r)$\\
	\hline
	47 & 3 & 279 & 282\\
	48 & 5 & 283 & 285\\
	49 & 7 & 287 & 291\\
	50 & 2 & 298 & 300\\
	52 & 6 & 306 & 312\\ 
	53 & 1 & 317 & 315\\
	54 & 3 & 321 & 321\\
	55 & 5 & 325 & 330\\
	\end{tabular}
\end{center}
In all cases, by Lemma~\ref{lem:ub}, $k_3(G) \le \lfloor M_5^*(m,8) \rfloor \le k_3(aK_{r+1}\cup\cC(b))$.\qedhere
\end{proof}

\section{Proof of Main Theorem}\label{sec:mainthm}

We now prove our main theorem.

\begin{main}If $G$ is a graph with $m$ edges and maximum degree at most $r$ for any fixed $r \le 8$, then $$k_3(G) \le k_3(aK_{r+1}\cup \cC(b)),$$ where $m = a\binom{r+1}{2}+b$ and $0 \le b < \binom{r+1}{2}$. That is, the graphs with the maximum number of triangles consist of as many disjoint copies of $K_{r+1}$ as possible, with the remaining edges formed into a colex graph.\end{main}

\begin{proof}
We will induct on $m$ and $r$. For $m \le \binom{r+1}{2}$, the Kruskal-Katona Theorem implies the theorem, so we may assume $m \ge \binom{r+1}{2}+1$. Let $G$ be an extremal graph.

By Corollary \ref{connected} and induction on $m$, if $G$ is not connected, then $k_3(G) \le k_3(aK_{r+1}\cup \cC(b))$, so we assume $G$ is connected. In particular, $G$ does not contain a $K_{r+1}$. By Corollary \ref{Delta=r} and induction on $m$ and $r$, $G$ contains a vertex of degree $r$. Therefore $G$ contains a cluster, and every cluster has at least one red edge and at least one blue edge.

For $r = 1$, every $G \in \cG(m,1)$ has $k_3(G) = 0 = k_3(aK_2\cup \cC(b))$. For $r=2$, we have assumed $G$ does not contain a $K_{r+1} = K_3$, so there are no triangles, and $G$ is not extremal. 

For $3 \le r \le 7$, Theorems \ref{thm:s=2} and \ref{thm:e=3} imply that every cluster of $G$ has $e(R) \ge \ceil{r/2}$. By Lemma \ref{lem:ub}, Lemma \ref{thm:seqopt}, and Theorem \ref{thm:seqopt}, $k_3(G) < f_3(m,r)$.

For $r = 8$, Theorems \ref{thm:s=2}, \ref{thm:e=3}, and \ref{thm:e=4} show that every cluster has $e(R) \ge 5$. Lemma \ref{lem:r=8} implies $k_3(G) \le k_3(aK_{r+1}\cup\cC(b))$. 
\end{proof}

\section{Open Problems}\label{open}

The question ``Which graphs with a fixed number of edges $m$ and maximum degree at most $r$ maximize the number of $K_t$'s?'' remains open for maximum degrees $r \ge 9$ and remains open for complete subgraph sizes $t \ge 4$. Similarly, the corresponding question when fixing the number of vertices $n$ instead of the number of edges remains open for $r \ge 7$ except when $a=1$. Both are extremely natural questions.

\bibliographystyle{plain}
\bibliography{Bibliography}

\begin{thebibliography}{10}

\bibitem{BBMGT}
B.~Bollobas.
\newblock {\em Modern Graph Theory}.
\newblock Graduate Texts in Mathematics. Springer New York, 1998.

\bibitem{CR2013}
J.~Cutler and A.~J. Radcliffe.
\newblock The maximum number of complete subgraphs in a graph with given
  maximum degree.
\newblock {\em J. Combin. Theory Ser. B}, 104:60--71, 2014.

\bibitem{E62}
P.~Erd{\H{o}}s.
\newblock On the number of complete subgraphs contained in certain graphs.
\newblock {\em Magyar Tud. Akad. Mat. Kutat\'o Int. K\"ozl.}, 7:459--464, 1962.

\bibitem{FFK88}
Peter Frankl, Zolt{\'a}n F{\"u}redi, and Gil Kalai.
\newblock Shadows of colored complexes.
\newblock {\em Mathematica Scandinavica}, 63(2):169--178, 1988.

\bibitem{Frohmader08}
A.~Frohmader.
\newblock Face vectors of flag complexes.
\newblock {\em Israel J. Math.}, 164:153--164, 2008.

\bibitem{G}
David Galvin.
\newblock Two problems on independent sets in graphs.
\newblock {\em Discrete Math.}, 311(20):2105--2112, 2011.

\bibitem{GLS}
W.~Gan, P.~S. Loh, and B.~Sudakov.
\newblock Maximizing the number of independent sets of a fixed size.
\newblock {\em Combin. Probab. Comput.}, 24(3):521--527, 2015.

\bibitem{H76}
N.~G. Had{\v{z}}iivanov.
\newblock A generalization of {T}ur\'an's theorem on graphs.
\newblock {\em C. R. Acad. Bulgare Sci.}, 29(11):1567--1570, 1976.

\bibitem{Hammer}
Peter~L. Hammer, Uri~N. Peled, and Xiaorong Sun.
\newblock Difference graphs.
\newblock {\em Discrete Appl. Math.}, 28(1):35--44, 1990.

\bibitem{K68}
G.~Katona.
\newblock A theorem of finite sets.
\newblock In {\em Theory of graphs ({P}roc. {C}olloq., {T}ihany, 1966)}, pages
  187--207. Academic Press, New York, 1968.

\bibitem{K63}
Joseph~B. Kruskal.
\newblock The number of simplices in a complex.
\newblock In {\em Mathematical optimization techniques}, pages 251--278. Univ.
  of California Press, Berkeley, Calif., 1963.

\bibitem{R}
Steven Roman.
\newblock The maximum number of {$q$}-cliques in a graph with no {$p$}-clique.
\newblock {\em Discrete Math.}, 14(4):365--371, 1976.

\bibitem{S71}
N.~Sauer.
\newblock A generalization of a theorem of {T}ur\'an.
\newblock {\em J. Combinatorial Theory Ser. B}, 10:109--112, 1971.

\bibitem{Z49}
A.~A. Zykov.
\newblock On some properties of linear complexes.
\newblock {\em Mat. Sbornik N.S.}, 24(66):163--188, 1949.

\end{thebibliography}
\end{document}